\documentclass[12pt,a4paper]{article}
\usepackage[utf8]{inputenc}
\usepackage{amsmath}
\usepackage{amsfonts}
\usepackage{shuffle}
\usepackage{amssymb}
\usepackage{graphicx}
\usepackage{mathabx}
  \usepackage{bbm}
        \usepackage{pdfsync}

\newcommand\sbullet[1][.5]{\mathbin{\vcenter{\hbox{\scalebox{#1}{$\bullet$}}}}}

\newtheorem{theorem}{Theorem}

\newtheorem{proposition}{Proposition}

\newenvironment{proof}[1][Proof]{\textbf{#1.} }{\ \rule{0.5em}{0.5em}}
\author{Yves Le Jan}
\title{On discrete loop signatures and Markov loops topology } 
\begin{document}
\maketitle
\begin{abstract}

Our purpose is to explore, in the context of loop ensembles on finite graphs, the relations between combinatorial group theory, the topology of loops, loop measures, and signatures of discrete paths. We determine the distributions of the loop homotopy class, and of the first and second homologies, defined by the lower central series of the fundamental group.

\end{abstract}
\footnotetext{ }
\footnotetext{  AMS 2000 subject classification:  .}
%\begin{abstract}
%We investigate the relations between the Poissonnian loop ensembles, their occupation fields, non ramified Galois coverings of a graph, the associated gauge fields, and random Eulerian networks.\\

%Notre étude montre les relations existant entre les ensembles poissoniens de lacets, les champs qu'ils définissent, les circuits euleriens, les revêtements galoisiens des graphes et les champs de jauges associés.
%\end{abstract}

%
\section{Introduction}
Our purpose is to explore, in the context of  on finite graphs, the relations between combinatorial group theory, loops and based loops, their signatures, loop measures and the associated loop ensembles.The main  object of study is the set of loops on a finite graph equipped with weights on vertices and edges.
This data allows us to define a Markov chain on the graph, a natural measure on loops and the corresponding Poissonian loop ensembles (Cf \cite{stfl}). 

This setup arises naturally in a wide range of contexts. Notably the measure can be seen as a discrete analogue of the measure defining the Brownian loop soup, (\cite{LW}, \cite {Symanz}). These Poissonian loop ensembles have also been the object of many investigations. Their properties can be studied in the context of rather general Markov processes, in particular Markov chains on graphs (Cf \cite{stfl}, \cite{lejanito}, \cite{lejanbakry} ). Two dimensional Brownian loop ensembles are a central object in the study of scaling limits of statistical mechanical systems at criticality.

In the first two sections following this introduction, we review some definitions (in particular the definition of a loop as equivalence class of based loops under the shift), relevant notions of algebraic topology for graphs and of combinatorial group theory. The correspondence between conjugacy classes of the fundamental group  and geodesic (i.e. non-backtracking) loops and the properties of the lower central series of the free groups lead us to introduce an appropriate notion of \emph{degree of a loop} and the associated element in the \emph{homology} space of same degree.\\

Section 4 introduces the notion of discrete signature of a based loop on the graph and explain its relation to the existing notion of path signature, due to Chen. This relation involves essentially a lift of the based loop to the universal abelian cover of the graph. The signature is introduced in the setting of free groups, following \cite{MKS}, then transferred to based loops.  As for paths, the signature is an element of the  tensor algebra  which can be represented by a power series. Its first non zero coefficient appears to depend only of the loop and determines its degree and its homology. In propositions 2, 3 and 4 of section 5, formulas are given to express the homologies in terms of the currents defined by multiple edge crossing numbers.\\

Probabilistic results start from section 6, in which we recall the definition of the  loop measure, and of the associated Poissonian loop ensembles. Reformulating known results about Selberg trace formula on graphs, we investigate the distribution of their homotopy classes, which are represented by closed geodesics on the graph.\\

In section 7, we determine the distributions of homologies
 of degree one and more generally of the holonomy defined by a G-connection, G being any finite group. These results  have mostly been obtained previously and are reviewed here for self-containedness. We recall that up to tree contour like parts,  loops are determined by the characters of their holonomies. \\

In section 8, using a discrete nilpotent holonomy group, an analogue of Schr\"odinger representation, and an inverse Fourier transform, we determine  in theorem 3 the distribution of the  homology of loops of degree two, and consequently  of the sum of the homologies of a  Poissonian loop ensemble of such loops.

%The purpose of the present work is to explore their topological properties 

\section{Geodesic loops and fundamental group for graphs}
% in particular the relation with gauge fields.
%We briefly recall the framework we use in the graph case. See \cite{stfl} for more details. Here we restrict our attention to discrete loops.
We consider a finite connected non-oriented graph $\mathcal{G}$, denoting by $X$ the set of vertices and by $E$ the set of edges. There is at most one edge between two distinct points and no loop edges\\
We will denote by $E^o$ the set of oriented edges, i.e. the set of pairs of points connected by an edge.\\
%The set of oriented edges is denoted $\overrightarrow{E}$.
 The extremities of an oriented edge $e$ are denoted $(e^{-},e^{+})$. The opposite oriented edge is denoted $-e$ and the corresponding non-oriented edge $\pm e$.
\\

\emph{Based loops} at $x$ are paths from $x$ to $x$. A shift $\theta$ acting on based loops is naturally defined as follow: If $l= (x_0, x_1,...x_n,x_0)$ is a based loop, $\theta(l)=(x_1, ...,x_n,x_0,x_1)$. \emph{Loops} are defined as equivalence classes of based loops up to change of base point, i.e.  more precisely as orbits under the shift action.\\

Recall that on graphs, geodesics paths are defined as non backtracking paths:  $(x_{0},x_{1},...,x_{n})$ with $\{x_{i},x_{i+1}\}$ in $E$ and $x_{i-1}\neq x_{i+1}$.For each choice of a base vertex $x$, the fundamental group (Cf \cite{Mass}) $\Gamma_{x}$ is defined by a set of based loops, the geodesic arcs from $x $ to $ x$, we will refer to as the \emph{geodesic based loops} with  base point $x$.The composition rule is concatenation followed by erasure of the backtracking subarc which can be produced by concatenation.\\
Every geodesic arc from $x_1$ to $x_2$ defines an isomorphism between $\Gamma_{x_1}$ and $\Gamma_{x_2}$. Each element of $\Gamma_{x_1}$ is mapped into an element of $\Gamma_{x_2}$ by concatenating on its right the geodesic arc from $x_2$  to $x_1$ and on its left the same arc traversed backwards. Backtracking subarcs have to be erased. As the induced bijection between conjugacy classes is independent of the chosen geodesic arc, we can identify them and denote by $Cl(\Gamma)$ this set of conjugacy classes.\\

 We say that based loop are tailless if their first edge and the opposite of their last edge differ (as in general, based loops have tails). The set of  tailless geodesic based loops is stable under the shift action (in contrast with the set of non-backtracking based loops). \emph{Geodesic loops} are defined as equivalence classes of non-backtracking tailless  based loops up to change of base point, i.e.orbits of non-backtracking tailless based loops under the shift action.   \emph{The set of geodesic loops is in canonical bijection with $Cl(\Gamma)$.}\\

 The groups  $\Gamma_{x}$ are all isomorphic to the free group $\Gamma(r)$ with $r=\vert E \vert-\vert X\vert+1$ generators, in a non canonical way.The isomorphisms between $\Gamma_{x}$'s and $\Gamma(r)$, and between different $\Gamma_x$'s  can be determined by the choice of a spanning tree $T$ of the graph. \\
Given the choice of a root vertex $x$, each oriented edge $\alpha$ outside the spanning tree defines an element $\gamma_{x}(\alpha)$ of  $\Gamma_{x}$ obtained by concatenation of the geodesic from $x$ to $\alpha^-$ in $T$, $\alpha$, and the geodesic from $\alpha^+$ to $x$ in $T$. We denote by $\gamma(\alpha)$ the corresponding element of the free group as it is clearly independent of the choice of the root $x$. Note that $\gamma(-\alpha)$ is the inverse of $\gamma(\alpha)$.\\ 

Finally recall that  each loop $l$ is homotopic to a unique geodesic loop $l^g$, obtained by removing from the loop all its backtracking parts, which are oriented contours of subtrees of the graph. In (almost) the same way, any based loop $l_{\sbullet}$ with base point $x_0$ is homotopic to an element $l_{\sbullet}^g$ of $\Gamma_{x_0}$. Note however that in this case the tail is not removed.\\
\section{The lower central series}
Recall (see section 5.3 in \cite{MKS}, section 6.4 in \cite{Reu}) that the lower central series $\Gamma^{(n)}$ of normal subgroups of a free group $\Gamma$  is defined recursively by setting $\Gamma^{(1)}=\Gamma$ and $\Gamma^{(n+1)}=[\Gamma^{(n)},\Gamma]$. The quotient groups $H_n={ \Gamma^{(n)}/ \Gamma^{(n+1)}}$ are abelian. Taking $\Gamma=\Gamma(r)$, the first homology group $H_1$ is the homology group of the graph, i.e the abelianized image of the fundamental group. By Witt's formula (Cf theorem 5-11 in \cite{MKS}), the $n$-th homology group $H_n$ is a free abelian group with $d_n=\frac{1}{n}\sum_{d\vert n}\mu(d) r^{n/d}$ generators, $\mu$ denoting the Moebius function. In particular, $d_1=r$, $d_2=\frac{r(r-1)}{2}$, and $d_3=\frac{r^3-r}{3}$.\\
The quotient group ${ \Gamma / \Gamma^{(n+1)}}$ is the free nilpotent group of class $n$, with $r$ generators (obtained by imposing that all iterated commutators of order $n+1$ vanish). \\Each element $u$ of $\Gamma^{(n)}$ projects on an element $h_n(u)$ of $H_n$ and as $h_n$ is a group homomorphism, it depends only on the conjugacy class of $u$. Indeed, for any $v$ in $\Gamma^{(n)}$, 
 $h_n(vuv^{-1})= h_n(v)+ h_n(u)- h_n(v)= h_n(u)$ .\\
 
We can now \emph{define the degree $d(\gamma)$ of a geodesic loop $\gamma$ as the highest index $n$ such that the associated conjugacy class is included in  $\Gamma^{(n)}$. We denote by $h(\gamma) $  the corresponding element of $H_{d(\gamma)}$.} Note that  $h_n(\gamma) $ vanishes for  $n< d(\gamma) $.

\section{Discrete signatures}
Given a spanning tree $T$ and an orientation  $e_j$ of the $r=\vert E \vert-\vert X\vert+1$ edges $\pm e_j$ not included in $T$, the elements $\gamma({e}_j)$  generate the free group  $\Gamma(r)$. These generators will be denoted by $\gamma_j$. They depend on the choice of the spanning tree but note that different spanning trees may produce the same set of generators.\\

Each group element  $g$  can be expressed by a word $\gamma_{j_1}^{n_1} \gamma_{j_2}^{n_2}...\gamma_{j_l}^{n_m}$, $m$ being a positive integer and the $n_i$ a $m$-tuple of non zero integers. We say that the word is reduced if consecutive $j$'s are distinct). \\
Each group element  $g$ (except the neutral element) can be expressed uniquely by a reduced word.\\
Cyclically reduced means reduced, and that either $j_1\neq j_m$, either that $j_1=j_m$ and $\text{sgn}(n_1)=\text{sgn}(n_m)$).\
Each conjugacy class of $\Gamma(r)$, or equivalently each geodesic loop, is represented uniquely by a class of cyclically reduced words equivalent under the shift.\\

Following the definition given in corollary 5.19 of \cite{MKS}, we will associate to such a goup element the formal series $ S(g)=\prod_{i=1}^{l}e^{n_{j_i}X_{j_i}}$, the $X_j$'s being non-commuting symbols.The set of formal series equipped with multiplication is naturally  isomorphic to the tensor algebra $T(\mathbb{R})$ generated by $\mathbb{R}^r$ i.e. the sum of tensor powers  $\mathbb{R}\oplus\bigoplus_{n=1}^{\infty}[\mathbb{R}^r]^{\otimes n}$ equipped with the tensor product.\\
By analogy with the definition of Chen \cite{Chen}, extended to bounded variation paths in \cite{Ly}, and in the theory of rough paths, we say that $S(g)$ is the \emph{signature} of $g$.\\
The signature of a based loop $l_{\sbullet}$,  with base point $x_0$, is defined as the signature of the corresponding geodesic based loop $l_{\sbullet}^g$, which is a group element of the fundamental group $\Gamma_{x_0}$.  

 Note that the connection with this notion becomes clear if first replace the graph by the associated cable graph, whose edges are copies of the unit interval. Choosing any base point $x_0$, one lifts the geodesic based loop defined by $g$ to a geodesic (i.e. non backtracking) path of the universal abelian cover of the graph, which is a graph on which $H_1$ acts freely and faithfully. Any spanning tree of the graph defines a tesselation of this cover in fundamental domains isomorphic to this tree. The universal abelian cover can be also extended to a cable graph.  By collapsing the spanning tree to a point, the graph becomes a bouquet, i.e a wedge sum of $r$ circles which are the images of the intervals associated with the $r$ edges $\pm e_j$. The cable graph defined by the universal abelian cover collapses into the a  cable lattice whose vertices are $\mathbb{Z}^r$ isomorphic to the standard lattice of parallel or orthogonal lines in $\mathbb{R}^r$.
 The lifted geodesic based loop defines a path on that lattice, starting at $0$ with endpoint in $\mathbb{Z}^r$. It depends on $g$ and $T$  but  does not in depend on the choice of $x_0$. \\
 Recall that the signature of a piecewise differentiable path in $\mathbb{R}^r$ is defined as the element of the tensor algebra $T(\mathbb{R}^r)$ whose coefficients are given by the corresponding iterated integrals along the path. (Cf \cite{Chen}  ).
Then, it follows directly from our definitions that: \\
\emph{The tensor algebra element corresponding to $S(g)$ is the signature of the corresponding path of the cable lattice}.\\
Performing a  time shift on based loops corresponds, on the abelian cover, to canceling the first increment of the path and adding it at the end the path so that the origin and the endpoint stay the same.\\

We denote by $\mathfrak{L}$ the free Lie algebra (Cf  section 5.6  i \cite{MKS}, or section 0.2 in \cite{Reu}  ) generated by the $X_j$'s. It is the space of formal series whose homogeneous terms of all  degrees are Lie polynomials.\\
  The sum of the terms of lowest degree in $S(g)-1$  is denoted $P_g(X)$.\\
 
Let us start by recalling a few fundamental properties.\\
For any word $u$ composed of  the non-commuting symbols $X_j$, we denote by $\shuffle$ the shuffle product of words (Cf \cite{Reu}, section 1.4), and $\langle S(g),u\rangle$ the coefficient of $u$ in $S(g)$. 
\begin{theorem}
For any $g$, $g_1$, $g_2$ in $\Gamma$,
\noindent
\begin{description}
\item  a) $S(g_1g_2)=S(g_1) S(g_2)$
\item b) $\log(S(g))$ belongs to $\mathfrak{L}$.
\item  c) The sum of the terms of lowest degree in $\log (S(g))$ is equal to $P_g(X)$.
 \item d) $P_g$ is a homogeneous Lie polynomial of degree $d(g)$.
 \item e)  $P_g$  depends only on the conjugacy class of $g$.
\item  f) For any pair of words $(u_1, u_2)$, $$\langle S(g), u_1 \shuffle u_2\rangle=\langle S(g),u_1 \rangle \;\langle S(g),u_2\rangle.$$ 
\end{description}
\end{theorem}
\begin{proof}

The product rule a) is obvious from our definition. For b), c), d) see  section 5.7 and corollary 5.19 in  \cite{MKS}.  As $S(g)S(g^{-1} )=1$, it follows easily from a)  that $P_{g^{-1}}(X)=-P_g(X)$ and that $P_g(X)$ and therefore $d(g)$) depends only on the conjugacy class $C(g)$. The shuffle identity f) is Theorem 2.5 in \cite{Ree}.
\end{proof}\\

\section{ Loop Homologies and Currents}

If $\gamma$ is a geodesic loop, we denote by $P_{\gamma}$ the Lie polynomial defined by the associated conjugacy class. Recall that its degree and homology were defined at the end of section 3. \\
If $l$ is a loop, we denote by $P_l$ the polynomial $P_{l^g}$ defined by the associated geodesic loop $l^g$ and by $d(l)$ its degree $d(l^g)$. The homology classes $h_n(l^g)$ will be denoted $ h_n(l)$ and the first non-zero homology class $h(l^g)=h_{d(l^g)}(l^g)$ will be denoted $ h(l)$.\\

As in \cite{stfl}, for any loop $l$ and any oriented edge $ e\in E^o$, we denote by $N_e(l)$ the number of times the  loop $l$ traverses $e$. More generally, we denote by   $N_e(\mathcal{L})=\sum_{l\in \mathcal{L}}N_e(l)$ the total number of traversals of  $e$ by a multiset of loops $\mathcal{L}$. The edge occupation field $N(\mathcal{L})$ verifies the Eulerian property:$$\sum_y N_{x,y}(\mathcal{L})=\sum_y N_{y,x}(\mathcal{L}).$$
 We can associate to each loop (or multiset of loops) a current $\widecheck{N}$ defined by $\widecheck{N}_{x,y}=N_{x,y}-N_{y,x}$ (or equivalently $\widecheck{N}_{e}=N_{e}-N_{-e}$) .  One checks easily from the Eulerian property that for any choice of spanning tree and outside edges orientation, $\widecheck{N}$ is determined by the $r$ integers $\widecheck{N}_{e_i}$. \
 Then we have 
\begin{proposition}
 \noindent
 Let us denote by  $\widecheck{N}_i$ the value of the current defined by a loop $l$ (or a multiset of loops $\mathcal{L}$) on the oriented edge $e_i$ outside the given spanning tree.
\begin{description}

\item a) $d(l)>1 $ iff $\widecheck{N}_i=0$ for all $1\leq i \leq  r$.
  \item b) If $d(l)=1$, $P_l=\sum_{1\leq i\leq r}\widecheck{N}_{i}  X_{i}$. 
  \item c)   If $d(l)\geq 1$, $ h_1(l)=\sum_{1\leq i\leq r}\widecheck{N}_{i} h_1(\gamma_i)$.
  
\end{description} 
 \end{proposition}
\begin{proof}
Note that $\widecheck{N}_{i}({l})=\widecheck{N}_{i}({l^g})$. Then compute the t term of degree 1 in $S(l^g)$ and $h_1(l)$, i.e. $h_1(l^g)$, using the decomposition of the signature in a product of exponentials stemming from the decomposition of $l^g$ by a cyclically reduced word.
\end{proof}\\

  For $l_{\sbullet}$ any based loop, we denote by $N_{e(1), e(2),...,e(m)}(l_{\sbullet})$ the number of increasing $m$-tuples of times at which $l_{\sbullet}$ crosses the $m$-tuple oriented edges $e(1)$...$e(m)$ successively. 
 %Denote by ${N}_{[i, j]}$ the number  $N_{e_i,e_j}-N_{e_j,e_i}$

\begin{proposition}\label{two}
 \noindent
  Denote by $ [X_{i} ,X_{j}]  $ the polynomial $ X_{i} X_{j}-  X_{j} X_{i}$,  by $[\gamma_i,\gamma_j]$ the commutator $\gamma_i^{-1}\gamma_j^{-1}\gamma_i \gamma_j$ and by  $\widecheck{N}_{i, j}$ the number ${N}_{e_i, e_j}+{N}_{-e_i, -e_j}-{N}_{e_i, -e_j}-{N}_{-e_i, e_j}$. 
 Then: 
\begin{description}
\item a) If $ h_1(l)$  vanishes (i.e. if $d(l)\geq 2$), and if $l_{\sbullet}$ is any based loop representing the loop $l$, for all  $1\leq i<j \leq r$, $\widecheck{N}_{i, j}(l_{\sbullet})-\widecheck{N}_{j, i}(l_{\sbullet})$ is independent of the base point and will therefore be denoted  $\widecheck{N}_{i, j}(l)-\widecheck{N}_{j, i}(l)$. It vanishes iff $d(l)>2$.
\item b)  If $ d(l)=2$, $P_l=\frac{1}{2}\sum_{1\leq i<j \leq r}(\widecheck{N}_{i,j}(l) -\widecheck{N}_{j,i}(l))  [X_{i} ,X_{j}]$. 
  \item c) If $d(l)\geq 2$,  $ h_2(l)=\frac{1}{2}\sum_{1\leq i<j \leq r}(\widecheck{N}_{i,j}(l) -\widecheck{N}_{j,i}(l))h_2([\gamma_i,\gamma_j])$.
  
 \end{description}
 \end{proposition}
\begin{proof}
Let $l_{\sbullet}$ be a loop based at $x_0$ representing $l$ and $l_{\sbullet}^g$ be the associated  geodesic based loop, defining an element $\gamma$ of $\Gamma_{x_0}$. Using the decomposition of the signature in a product of exponentials, we get that the sum of the  terms of degree $\leq 2$ in $S(\gamma)$ equals:\\
\begin{align*}
1&+\sum_{1\leq i\leq r}\widecheck{N}_{i}(l)  X_{i}\\
  &+\sum_{1\leq i\leq r}\frac{1}{2}(N_{e_i}(l)(N_{e_i}(l)-1)+N_{-e_i}(l)(N_{-e_i}(l)-1)-2N_{e_i}(l)N_{-e_i}(l))  X_{i}^2\\
  &+\sum_{1\leq i<j \leq r}(\widecheck{N}_{i,j}(l)  X_{i} X_{j}
   +\widecheck{N}_{j,i}(l)  X_{j} X_{i})\\
   &+\frac{1}{2}\sum_{1\leq i\leq r}(N_{e_i}(l)+N_{-e_i}(l))  X_{i}^2. 
   \end{align*}

   The second and third sum come from the products of terms of degree one in two exponentials and the last sum from the terms of degree two in one exponential.\\
   Hence, the sum of the  terms of degree $\leq 2$ in $S(\gamma)$ equals:

$$1+\sum_{1\leq i\leq r}\widecheck{N}_{i}(l) X_{i} +\frac{1}{2}[\sum_{1\leq i\leq r}\widecheck{N}_{i}(l) X_{i}]^2 +\frac{1}{2}\sum_{1\leq i<j \leq r}(\widecheck{N}_{i,j}(l) [ X_{i}, X_{j}]+\widecheck{N}_{j,i}(l) [ X_{j} ,X_{i}]).$$ 
Note also that  the terms of degree $\leq 2$ in $\log(S(\gamma))$ are $$\sum_{1\leq i\leq r}\widecheck{N}_{i}(l)  X_{i} +\frac{1}{2}\sum_{1\leq i<j \leq r}(\widecheck{N}_{i,j}(l)-\widecheck{N}_{j,i}(l)) ) [ X_{i} X_{j}].$$ a) and b) follow directly, by Theorem 1. Then c) follows from Theorem  5.12 in \cite{MKS} and its corollary. 
\end{proof}

\medskip

%$\widecheck{N}_{e_i,e_j,e_k}(l)=\sum_{\sigma\in \mathcal{S}_3}
%\varepsilon(\sigma)N_{e_i,e_j,e_k}(l)$. 
More generally, we can define $$\widecheck{N}_{{i(1)}, {i(2)},...,{i(m)}}(l_{\sbullet})=\sum_{\epsilon_k=\pm , 1\leq k \leq m\;}\prod_{k=1}^m\epsilon_k \;N_{\varepsilon_1e_{i(1)}, \varepsilon_2e_{i(2)},...,\varepsilon_me_{i(m)}}(l_{\sbullet}).$$  
Then, if $d(l)=m$, it follows directly from its definition that:

$$P_l=\sum_{i(1), i(2),...,i(m)}\widecheck{N}_{i(1), i(2),...,i(m)}(l_{\sbullet})X_{i(1)}X_{i(2)}...X_{i(m)}$$ for any representative $l_{\sbullet}$ of $l$,  Moreover from theorem 1.4  in\cite{Reu} (or in Theorem 5.17 in \cite{MKS} ) this last expression can be rewritten as follows:
\begin{proposition}
$$P_l =\frac{1}{m}\sum_{i(1), i(2),...,i(m)}\widecheck{N}_{i(1), i(2),...,i(m)}(l_{\sbullet})[[...[X_{i(1)},X_{i(2)}]...]X_{i(m)}].$$

\end{proposition}
\emph{Remark:} Equivalently, $P_l$ equals: $$\frac{1}{m}\sum_{i(1)<i(2),i(3),...,i(m)}(\widecheck{N}_{i(1), i(2),...,i(m)}(l_{\sbullet})-\widecheck{N}_{i(2), i(1),i(3),...,i(m)}(l_{\sbullet}))[[...[X_{i(1)},X_{i(2)}]...]X_{i(m)}].$$
This gives a non self-contained proof of proposition \ref{two}. In general, this expression can be further modified to get a decomposition in any specific basis of the free Lie algebra (cf \cite{Reu}, chapter 4). Then it follows that its coefficients which are linear combinations of  $\widecheck{N}$'s depend only on $l$, as $\widecheck{N}_i$, and $\widecheck{N}_{i,j}$ if $d(l)>1$. 
%One can check that if $d(l)=3$, $P_l=\sum_{i>j, k\neq i,j}(\widecheck{N}_{e_i,e_j,e_k}(l)+\widecheck{N}_{e_k,e_j,e_i}(l)-\widecheck{N}_{e_j,e_i,e_k}(l)-\widecheck{N}_{e_k,e_i,e_j}(l))[[X_i,X_j]X_k]+\sum_{i\neq j}(\widecheck{N}_{e_i,e_j,e_i}(l)+\widecheck{N}_{e_i,e_j,e_i}(l)-\widecheck{N}_{e_j,e_i,e_i}(l)-\widecheck{N}_{e_i,e_i,e_j}(l))[[X_i,X_j]X_i] $.

For example, if $d(l)=3$, using Jacobi identity, we get that $$ [[X_k,X_i]X_j]=-[[X_i,X_j]X_k]-[[X_j,X_k]X_i].$$ 
%Morevover, as $\widecheck{N}_k(l)=0$, $\widecheck{N}_{i,j,k}+\widecheck{N}_{j,k,i}+\widecheck{N}_{k,i,j}=0$.
 Hence 
 \begin{align*}
 P_l=&\frac{1}{3}\sum_{i<j< k}
 [\widecheck{N}_{i,j,k}-\widecheck{N}_{j,i,k}-\widecheck{N}_{k,i,j}+\widecheck{N}_{i,k,j}](l)[[X_i,X_j]X_k]\\
+&\frac{1}{3}\sum_{i<j<k}[\widecheck{N}_{j,k,i}-\widecheck{N}_{k,j,i}-
 \widecheck{N}_{k,i,j}+\widecheck{N}_{i,k,j}](l)[[X_j,X_k]X_i]\\
+&\frac{1}{3}\sum_{i\neq j}[\widecheck{N}_{i,j,i}-\widecheck{N}_{j,i,i}](l)[[X_j,X_i]X_i].
  \end{align*}
Moreover, as $\widecheck{N}_j(l)=0$, $\widecheck{N}_{i,j,k}+\widecheck{N}_{i,k,j}+\widecheck{N}_{j,i,k}=\widecheck{N}_j\widecheck{N}_{i,k}=0$ and  $\widecheck{N}_{j,k,i}+\widecheck{N}_{k,i,j}+\widecheck{N}_{k,j,i}=0$. Therefore:
\begin{align*}
P_l&=\frac{1}{3}\sum_{i<j< k}([(\widecheck{N}_{j,i,k}+2\widecheck{N}_{k,i,j}](l)[[X_j,X_i]X_k]\\ &+\sum_{i<j< k}[(2\widecheck{N}_{j,k,i}+\widecheck{N}_{i,k,j}](l)[[X_j,X_k]X_i])\\
&+\frac{2}{3}\sum_{i\neq j}\widecheck{N}_{i,j,i}(l)[[X_j,X_i]X_j]
\end{align*}
This decomposition in $2\frac{r(r-1)(r-2)}{6}+ r(r-1)=\frac{r(r^2-1)}{3}$ terms corresponds to the choice of a basis in the second homology space. We obtain:
\begin{align*}
h_3(l)l&=\frac{1}{3}\sum_{i<j< k}[(\widecheck{N}_{j,i,k}+2\widecheck{N}_{k,i,j}](l)h_3([[\gamma_j,\gamma_i]\gamma_k])\\ &+\sum_{i<j< k}[2\widecheck{N}_{j,k,i}+\widecheck{N}_{i,k,j}](l)h_3([[\gamma_j,\gamma_k]\gamma_i])\\
&+\frac{2}{3}\sum_{i\neq j}\widecheck{N}_{i,j,i}(l)h_3([[\gamma_j,\gamma_i]\gamma_j]).
\end{align*}.

% There are $\frac{r^3-r}{3}$ terms in this sum and the $\frac{r^3-r}{3}$ double brackets in it form a basis of the space of free Lie polynomial of degree 3.

%$H_1(\mathcal{G},\mathbb{Z})$
%\begin{figure}
% \includegraphics[width=6cm] {Erasure.pdf}
% \caption{Loop $\rightarrow$ Geodesic Loop}
% \end{figure}
\section{Loop measures and homotopies distribution.} 
 Following  \cite{stfl}, we attach a positive conductance $C_e$ to each edge $e\in E$ and a killing rate $\kappa_x$ to each vertex $x\in X$, then we define the duality mesure  $\lambda_{x}=\kappa_{x}+\sum_{y}C_{x,y}$
and the 
 $\lambda$-symmetric transition matrix $P$:  $P_y^x=\dfrac{C_{x,y}}{\lambda_x}$, $P_{\Delta}^x=\kappa_{x}{\lambda_x}$.
 The energy functional is:
$$\epsilon(f,f)=\frac{1}{2}\sum_{x,y}C_{x,y}(f(x)-f(y))^{2}+\sum_{x}\kappa_{x}f(x)^{2}$$
and the associated Green matrix $(diag(\lambda)-C)^{-1}$ is denoted by $G$.\\
By definition, the multiplicity of a based loop is the maximal  integer $m$ such that it is the concatenation of $m$ identical based loops. As it does not depend on the choice of the base point, it is in fact defined on the set of loops. Denoting $\text{mult}(l)$ the multiplicity of the loop $l$, we define a measure $\mu$ on loops:  $$\mu(l)=\frac{1}{\text{mult}(l)}\prod_{ \text{edges of } l} \left( P^{e^{-}}_{e^{+}}\right).$$
  Recall that the mass of  $\mu$ equals
$-\log (\det(I-P))$.\\

 Note that this measure is induced by the restriction to non-trivial discrete loops of the measure $\sum_{x\in X}\int_{0}^{\infty}\frac{1}{t}\,\mathbb{P}_{t}^{x,x}\lambda_x dt$ defined on continuous time based loops, $\mathbb{P}_{t}^{x,x} $ being the non-normalized bridge measure defined by the transition semigroup $\exp(t[I-P])$ associated with the energy functional (Cf \cite{stfl}). It is the discrete space version of the loop measure defined by Lawler and Werner (\cite{LW} and Symanzik (\cite{Symanz} ). \\

We denote by $\mathcal{L}_{\alpha}$ the Poisson point process of intensity $\alpha \mu$.
The ensemble $\mathcal{L}_{\alpha}$ can be decomposed into independent muitisets of loops of distinct homotopies: For any geodesic loop $\gamma$, the number of loops $l\in  \mathcal{L}_{\alpha}$ such that $l^g=\gamma$ is a Poisson variable of parameter $\mu_{\gamma}$.\\ 
In the case of the regular graphs with unit conductances and constant $\kappa$, a simple expression of $\mu_{\gamma}$ is obtained directly from the results of section 4.2 of \cite{Mn}: 

\begin{proposition}
If $\mathcal{G}$ is a $d$-regular graph, with $C_e=1$, $\kappa$ constant, for any closed geodesic $\gamma$, the number of loops homotopic to $\gamma$ is a Poisson r.v. of expectation:$$\mu_{\gamma}=\frac{1}{\text{mult}(\gamma)} \left( \frac{d+\kappa}{2(d-1)}( 1-\sqrt{1-\frac{4(d-1)}{(d+\kappa)^2}}) \right) ^{\vert \gamma \vert} $$

In particular, for $\kappa=0$, $\mu_{\gamma}=\frac{1}{\text{mult
}(\gamma)}(d-1)^{-\vert \gamma\vert}$
\end{proposition}
\emph{Remark:} If $\kappa=\frac{1}{u}+u(d-1)-d$, the associated generating function $\sum_{\gamma}u^{\vert \gamma \vert}\mu_{\gamma}=\sum_{\gamma}\frac{1}{\text{mult}(\gamma)}u ^{\vert \gamma \vert}$
%=\sum \frac{1}{\text{mult}(\gamma)} \left( \frac{u(d+\kappa)}{2(d-1)}( 1-\sqrt{1-\frac{4(d-1)}{(d+\kappa)^2}}) \right) ^{\vert \gamma \vert}$$ 
coincides with the logarithm of Ihara's zeta function (Cf \cite{Starter},\;\cite{Kotasun}\, \cite{stfl}): $$\sum_{\gamma}\frac{1}{\text{mult}(\gamma)}u ^{\vert \gamma \vert}=\log(\zeta_{Ih}(u))=-\log[(1-u^{2})^{-\chi}\det(I-uA+u^{2}(d-1)I)]$$ where $A$ denotes the adjacency matrix and $\chi$
 the Euler number $\left|  E\right|  -\left|  X\right|  $ of the graph.\\

The result stated in the previous proposition follows from a more general one:
If $(x,y)$ is an edge,  let us denote $r^{x,y}$ the probability that the Markov chain starting at $y$ returns to $y$  following a tree-contour subloop without visiting $x$ at the first step. Note that:
  $$r^{x,y}=\sum_{z\neq x}P^{y}_z P^{z}_y \sum_{n=0}^{\infty}[r^{y,z}]^n$$
  and, if we set $\rho^{x,y}=\sum_{n=0}^{\infty}[r^{x,y}]^n$, 
$$\rho^{x,y}=1+\sum_{z\neq x}P^{y}_z  \rho^{y,z}P^{z}_y \rho^{x,y} $$
  We get the following:\\
   \begin{proposition} If $\gamma$ varies in the set of geodesic loops (i.e. the set conjugacy classes), $\vert\{l\in \mathcal{L},\, l^g=\gamma\} \vert$ are independent Poisson r.v. with mean values
$$\mu_{\gamma}=\frac{1}{\text{mult}(\gamma)}(\prod_{ \text{edges of } \gamma}P_{e^+}^{e-} \rho^{e^-,e^+})^{\text{mult}(\gamma)} $$
\end{proposition} 
If $\mathcal{G}$ is a $d$-regular graph, with $C_e=1$, $\kappa$ constant, lifting the  tree contour subloop to the universal cover of the graph, (which is a $d$-regular tree) we see that the $r^{x,y}$ are all equal to the return probability of a random walk starting from the root of the half- $d$-regular tree obtained by cutting the edge $\{x, y\}$ in the universal cover.   We then get from the quadratic equation satisfied by $\rho$ that $$\rho^{x,y}=\frac{(d+\kappa)^2}{2(d-1)}\left( 1-\sqrt{1-\frac{4(d-1)}{(d+\kappa)^2}}\right) $$ and recover the result of \cite{Mn}.\\
This argument is close to the proof of Ihara's formula in \cite{Starter}.\\
The corresponding generalization of Ihara's formula is given in \cite{An}.
A different generalization was given in \cite{WF}.\\

\medskip
Let us now consider the distribution of the number of loops homotopic to a point; It is obviously a Poisson distribution of parameter $$ -\ln(\det(I-P)-\sum_{\gamma}\mu_{\gamma}.$$
  To compute this quantity, let us now denote $r^{x,y,k}$ the probability that the Markov chain starting at $y$ returns to $y$ for the first time in $2k$ steps following a tree-contour subloop without visiting $x$. Set $r^{x,y}(s)=\sum r^{x,y,k}s^k$. Set $\rho^{x,y}(s)=\sum_{n=0}^{\infty}[r^{x,y}(s)]^n$\\  
 Note that:
  $$r^{x,y}(s)=s\sum_{z\neq x}P^{y}_z P^{z}_y \rho^{y,z}(s)$$
  and that $\rho^{x,y}(s)$ satisfies the relation:
$$\rho^{x,y}(s)=1+s\sum_{z\neq x}P^{y}_z  \rho^{y,z}(s)P^{z}_y \rho^{x,y}(s).$$

Let us now denote $r^{x,k}$ the probability that the Markov chain starting at $x$ returns to $x$ for the first time in $2k$ steps following a tree-contour subloop. Set $r^{x}(s)=\sum r^{x,k}s^k$ Note that:
  $$r^{x}(s)=s\sum_{y}P^{x}_y P^{y}_x \rho^{x,y}(s).$$
 
  Let  denote $\rho^{x,k}$ the probability that the Markov chain starting at $x$ returns to $x$ in $2k$ steps (not necessarily for the first time) following a tree-contour subloop. Set $\rho^{x}(s)=\sum_0 ^\infty \rho^{x,k}s^k$ and note that:  $$\rho^{x}(s) = \frac{1}{1-r^{x}(s)}=\frac{1}{1-s\sum_{y}P^{x}_y P^{y}_x \rho^{x,y}(s)}$$. \\
  %Hence   $$\dfrac{\rho^{x}(s)}{1+\rho^{x}(s)}=r^{x}(s)=s \sum_{y}P^{x}_y P^{y}_x (1+\rho^{y}(s))$$
The number of loops of $\mathcal{L}$  homotopic to a point is a Poisson r.v. with expectation $$\sum_x \sum_1 ^\infty\frac{1}{2k} \rho^{x,k} =\sum_x \int_0^1\frac{\rho^{x}(s)-1}{2s}ds.$$
\begin{proposition} If $\mathcal{G}$ is a $d$-regular graph, with $C_e=1$, $\kappa$ constant, the number of loops homotopic to a point is a Poisson r.v. of expectation $$\vert X \vert [\frac{d}{2}( \log(2)-\log(b+1))+ (d-2) (\log(b+\frac{d-2}{d})-\log(1+\frac{d-2}{d}))]$$ with $b= \sqrt{1-4\frac{d-1}{(d+\kappa)^2}}$.\\
In particular, for $\kappa=0$, $b=\frac{d-2}{d}$ and this is equal to: $\vert X \vert[(d-2)\log (d-2)+(d/2)\log (d)-(d-2+d/2)\log(d-1)]$.
\end{proposition}
%$$\mu(\gamma)=\sum_{l,\, l^o=\gamma}\mu(l)$$int_0^1\frac{\rho^{x}(s)}{s}ds$
\begin{proof}

 It is clear that $\rho^{x,y}(s)$ and $\rho^{x}(s)$ are constants in the edge or vertex variables. We get from the previous equations that $$\rho^{x,y}(s)=\frac{(d+\kappa)^2}{2s(d-1)}(1-\sqrt{1-\frac{4s(d-1)}{(d+\kappa)^2}})$$
and
 $$\rho^{x}(s)=\frac{2(d-1)}{d-2+d\sqrt{1-\frac{4s(d-1)}{(d+\kappa)^2}})}.$$
 For $\kappa=0$, this is, up to a change of notations, formula (19) in section 4.1 of\cite{Mn}.\\
%and recover the result of \cite{Mn}.\\
%  $\int_0^1(\rho(s)-1)\frac{1}{2s}ds$ avec $\rho(s)=\frac{2(d-1)}{d-2+d\sqrt{1-\frac{4s(d-1)}{(d+\kappa)^2}})}$ $d>2$ et $\kappa$ sont des paramètres. On prendra d'abord $\kappa=0$\\
\medskip
From the expression of $\rho^x$, by an elementary integration, we finally deduce the proposition.
\end{proof}

%(dl)\lambda_{x}dt

%A probability measure $\nu$ is defined on spanning trees (Cayley): $$\nu(T)= \frac{\prod_{ \text{edges of }T }P^{e^{-}}_{e^{+}}}{\det(I-P) }$$

\section{First Homology and  holonomies distributions}
The loops of $\mathcal{L}_{\alpha}$ can also be classified into independent sets of loops $\mathcal{L}_{\alpha}^{(d)}$ according to their degrees $d(l^g)$. For each degree $d$, we can try to determine the distribution of the sum of the homologies of the loops of degree $d$ and the distribution of the number of loops of given $d$-th homology,\\

In this section, we recall and complete the results obtained in \cite{stfl}, \cite{lejanito}, and \cite{lejanbakry}, which solve the problem for $d=1$.

Define for any oriented edge $(x,y)$ and integer $1\leq i \leq r$,  $ \eta^i_{x,y}= 1_{x=e_i^-, y=e_i^+}-1_{y=e_i^-, x=e_i^+}  $.\\
Let $\theta_1, ... ,\theta_r$ be $r$ real parameters. Denoting by $P^{(\theta)}$ the matrix $P_y^xe^{2\pi \sqrt{-1}\sum \theta_i \eta^i_{x,y}}$, we have (see \cite{stfl}):
$$\int(e^{\sum_{i}{\sqrt{-1}\pi} \widecheck{N}_{i}(l) \theta_{i} }\
-1)\mu(dl)=-\log(\det(I-P^{(\theta)}).$$
Hence for any $(h_i)\in \mathbb{Z}^r$, using an inverse Fourier transform, we have: 
\begin{theorem}

The integers  $\vert\{l\in \mathcal{L}_{\alpha},\;\widecheck{N}_{i}(l)=h_i, \; i=1...r \} \vert$ are independent Poisson r.v. with expectations:
 $$\alpha\mu(\{l,\;\widecheck{N}_{i}(l)=h_i, \; i=1...r\})=-\alpha
\int_{[0,1]^r}\log(\det(I-P^{(\theta)}))\prod_{i =1}^r e^{-2\pi \sqrt{-1}\, h_i \theta_i}d\theta_i.$$\\
%$\sqrt{\det(J)}$ being the volume of the Jacobian torus $\sqrt{\det[\delta_{i,j}C_{e_{i}}-C_{e_{i}}K_{e_i,e_j}C_{e_{j}}]}.
\end{theorem}

\emph{Remarks:} \\
- Consequently, the distribution of the homology field defined by $\mathcal{L}_{\alpha}$ is:
  $$P(\widecheck{N}_{i}(\mathcal{L}_{\alpha})=h_i, \; i=1...r)=\int_{[0,1]^r}\left[\frac{\det(I-P)}{\det(I-P^{(\theta)})}\right]^{\alpha} \prod_{i =1}^re^{-2\pi \sqrt{-1} ,h_i \theta_i}d\theta_i$$\\ 
- An intrinsic, but less explicit, expression ( not relying on the choice of the spanning tree) is given in \cite{lejanito}. It involves discrete differential forms  (Cf \cite{stfl}, section 1-5). The Fourier integration is done on the Jacobian torus (\cite{kosun}), i.e. the quotient of the space of harmonic one-forms $H^1(\mathcal{G},\mathbb{R})$ (i.e.  the space of one-forms $\omega$ such that $\sum_{y}C_{x,y}\omega^{x,y}=0$ for all $x\in X$)  by $H^1(\mathcal{G},\mathbb{Z})$ the space of harmonic one-forms with  $\mathbb{Z}$-valued integrals on loops. The Lebesgue measure is normalized by its volume which is equal to $\sqrt{\det (J)}$, with $J_{i,j}=\delta_{i,j}C_{e_{i}}-C_{e_{i}}K_{e_i,e_j}C_{e_{j}}$, for $1\leq i,j\leq r$, $K$ denoting the transfer matrix: $K_{e,f}=G_{e^+,f^+}+G_{e^-,f^-}-G_{e^+,f^-}-G_{e^-,f^+}$.\\
- For $\alpha=1$, an alternative expression in terms of Bessel functions (and without inverse Fourier transform) is given in section 3 of \cite{flownote}.\\
 % (= sum of the conductances of all the spanning trees of the graph).

To try to to solve the problem for higher values of $d$, in particular for $d=2$, we need to recall more results.

Morphisms from the  fundamental groups $\Gamma_x$ in a group $G$ can be obtained from
maps $A$, assigning to each oriented edge $ e $ an element  $A[e]$ in $G$ with $A[-e]=A[e]^{-1}.$\\
A path, in particular a based loop, is mapped to the product of the images by $A$ of its oriented edges and the associated loop  $l$ to the conjugacy class of this product, i.e. the holonomy of $l$, is denoted $H_A(l)$. Moreover this holonomy  is unchanged if we replace the loop $l$ by the geodesic loop $l^g$ homotopic to $l$.

 A gauge equivalence relation between such assignment maps is defined as follows:  $A_1\sim A_2$ iff  there exists a map $ Q$: $X\mapsto G$ such that:
 $$A_2[e]=Q(e^{+})A_1[e]Q^{-1}(e^{-})$$
Equivalence classes are $G$-connections. They define $G$-Galois coverings of $\mathcal{G}$ (cf \cite{lejanbakry}). Obviously, holonomies depend only on connections.\\
 Given a spanning tree $T$, there exists a unique  $A^T\sim A $ such that $A^T[e] =I$ for every edge $e$ of $T$.\\ For any unitary representation $\pi$ of $G$, denote $ \chi_{\pi}(C)$ the \emph {normalized trace} of the image by $\pi$ of any element in the conjugacy class $C$.\\
 Recall that free groups are \emph{conjugacy separable:}
Two conjugacy classes are separated by a morphism in some \emph{finite} group.\\  
Conjugate separability implies that if we consider all unitary representations of finite groups and all connections, the holonomies determine the geodesic loop (i.e. the conjugacy class of $\Gamma$) defined by $l$. The functions $\gamma \mapsto  \chi_{\pi}(H_A(\gamma))$ span an algebra and separate geodesic loops.

 \bigskip
 Fix now a finite group $G$, and let $ \mathcal{R}$ denote the set of irreducible unitary representations of $G$.\\
 Define an extended transition matrix $P^{A,\pi}$ with indices in  $X\times \{1, 2,...\dim(\pi)\}$ by $[P^{A,\pi}]^{x,i}_{y,j}=P^x_y [\pi(A[(x,y)])]^i _j$. Then the following proposition follows directly from the expression of the based loop measure inducing $\mu$ (see \cite{stfl}):
 \begin{theorem}
 
 $$\sum_l \chi_{\pi}(H_A(l)) \mu(l)=-\frac{1}{\text{dim}(\pi)}\log(\det(I-P^{A,\pi}))$$ 
\end{theorem}

\emph{Remarks:} \\
a) This result extends to compact groups.\\
b) For any unitary representation $\pi$, choose, for any oriented edge $e_j$, an Hermitian matrix $H^{(\pi)}_j$, such that $\exp[\sqrt{-1}H^{(\pi)}_j]=\pi[A(e_j)]$. For any based loop representative of $l$, denoted $l_{\sbullet}$, let the holonomy can be expressed as the normalized trace of  the signature series acting on the matrices $H^{(\pi)}_j$ in place of the $X_j$'s: $$ \chi_{\pi}(H_A(l)) =\frac{1}{\text{dim}(\pi)}Tr(S(l_{\sbullet}^g)[H_j, 1\leq j\leq r].$$
c) It follows from this proposition and group representation theory (Cf for example \cite{Zag}) that  $\vert\{l\in \mathcal{L}_{\alpha},\,H_A( l)= C \} \vert$ are independent Poisson r.v. with expectations:
 $$\alpha\mu(\{l,  \:H_A( l)= C\})=-\alpha\sum_{\pi \in \mathcal{R} }\overline{ \chi_{\pi}(C)}\frac{\vert C \vert}{\vert G \vert}\text{dim}(\pi) \log(\det(I-P^{A,\pi})).$$

% \end{figure}
\section{Nilpotent holonomy and second homology distribution}

 Let us now consider the case where $G$ is a nilpotent group of class 2 based on the field $\mathcal{Z}_p=\mathcal{Z}/ p\mathcal{Z}$, for some prime number $p$. This group is defined as follows:\\
 
 $G= \{(a,c), \; a\in \mathbb{Z}_p^r, c\in [\mathbb{Z}_p^r]^{\wedge 2} \}$, with product :
 $$(a,c)\cdot(a',c')=(a+a',c+c' +\frac{1}{2}(a\otimes a'- a'\otimes a)).$$Note that $\mathbb{Z}_p^r]^{\wedge 2}$ can be identified with the set of $(r,r)$ skew-symmetric matrices with coefficients in $\mathbb{Z}_p$.\\
 Associativity is checked easily. The neutral element is $(0,0)$, and $(a,c)^{-1}=(-a, -c)$.\\
 Note also that $(a',c')\cdot (a,c)\cdot (-a',-c')=(a,c-(a\otimes a'- a'\otimes a))$. Hence the set of conjugacy classes can be identified with the set of pairs $(a,v), \; a\in \mathbb{Z}_p^r, v\in [\mathbb{Z}_p^r]^{\wedge 2}/a\wedge \mathbb{Z}_p^r$.
 
 For any $(r,r)$ skew-symmetric matrix $h_{i,j}$ with coefficients in $\mathbb{Z}_p$, a unitary representation $U_h$ of $G$ on the space $V_{r,p}$ of functions on $\mathbb{Z}_p^r$ is defined as follows:
 $$U_h[(a,c)]\psi (x)=e^{\frac{2\pi \sqrt{-1}}{p}(\left\langle c,h\right\rangle+\left\langle a,x\right\rangle )}\psi(x-h\cdot a)$$
 with $\left\langle c,h\right\rangle=\sum_{1\leq i,j \leq r}h_{i,j}c_{i,j}$, $(h\cdot a)_i=\sum_{1\leq j \leq r} h_{i,j}a_j$  and $\left\langle a,x\right\rangle =\sum_{1\leq i \leq r} a_{i}x_{i}$ (this is similar to the Schrödinger representation of the Heisenberg group).\\
 Indeed, noting that $\left\langle a',h\cdot a\right\rangle=-\left\langle h,\frac{1}{2}(a\otimes a'- a'\otimes a)\right\rangle $,  $$U_h[(a,c)]U_h[(a',c'])\psi (x)=e^{\frac{2\pi \sqrt{-1}}{p}(\left\langle c+c',h\right\rangle+\left\langle a',x\right\rangle +\left\langle a,x-h\cdot a'\right\rangle )}\psi(x-h\cdot a-h\cdot a')$$$$=U_h[(a,c)(a',c')]\psi (x).$$
 The dimension of $(V_{r,p})$ is $p^r$ and an orthonormal basis of $V_{r,p}$ is given by products of  exponentials $\psi_{n_1, ...,n_r}(l_1,...l_r)=e^{\frac{2\pi \sqrt{-1}}{p}\sum_{1\leq i\leq r}l_i n_i}$, with $0\leq l_i<p$.
 We can check that the normalized trace $\chi_{U_h}((a,c))$ equals $1_{\left\lbrace a=0 \right\rbrace }e^{\frac{2\pi \sqrt{-1}}{p}\left\langle c,h\right\rangle}$.\\
 
 Consider the $G$-connection $A$ defined by assigning  to each edge $e_i,\; i\in \{1,...r\}$ the element $(v_i,0)$, $v_i$ being the $i$-th element of the canonical basis of $\mathbb{R}^r$.

Then if $l_{\sbullet}$ is any based loop representative of  $l$, it defines an element of $G$: $((\widecheck{N}_i(l),1\leq i\leq r), (\widecheck{N}_{i,j}(l_{\sbullet}),1\leq i<j \leq r))$ which is a representative of $H_A(l)$ in $G$. Indeed, it is enough to prove it for a geodesic based loop $\gamma$, by induction on the size of the loop. Noting that $\widecheck{N}_{i,j}(\gamma \gamma_{j}^{\pm1})=\widecheck{N}_{i,j}(\gamma) \pm\widecheck{N}_i(\gamma)$ and $\widecheck{N}_{i,j} (\gamma \gamma_{q}^{\pm1})=\widecheck{N}_{i,j}(\gamma) $ if $q\neq i,j$, one checks easily that $((\widecheck{N}_i(\gamma\gamma_{k}^{\pm1}))), (\widecheck{N}_{i,j}(\gamma \gamma_{k}^{\pm1})))=((\widecheck{N}_i(\gamma)), (\widecheck{N}_{i,j}(\gamma)))(\pm v_k,0)$. Therefore $$\chi_{U_h}(H_A(l))=1_{ \{ l,\;\widecheck{N}_i(l)=0, \forall\;i  \}  }e^{\frac{2\pi \sqrt{-1}}{p}\sum_{1\leq i<j \leq r}\widecheck{N}_{i,j}(l)h_{i,j}}$$ (note we can write $\widecheck{N}_{i,j}(l)$ instead of  $ \widecheck{N}_{i,j}(l_{\sbullet}))$ as the first homology vanishes).\\
 By the previous proposition, $$\sum_l \chi_{U_h}(H_A(l) \mu(l)=-\frac{1}{p^r} \log(\det(I-P^{A,U_h})).$$
Hence, for any $(r,r)$ skew-symmetric matrix $u_{i,j}$ with coefficients in $[0\, ,\,1)$,

 $$\sum_{\{ l,\;h_1(l)=0 \} }e^{2\pi \sqrt{-1}\langle h_2(l),u\rangle} \mu(l)=\sum_{\{ l,\;\widecheck{N}_i(l)=0, \forall\;i \} }e^{2\pi \sqrt{-1}\sum_{1\leq i<j \leq r}\widecheck{N}_{i,j}(l)u_{i,j}} \mu(l)$$
 and denoting by $h(u,p)_{i,j}$ the integral part of $u_{i,j}p$
$$=\lim_{p\uparrow \infty}\sum_{\{ l,\;\widecheck{N}_i(l)=0, \forall\;i \} }e^{\frac{2\pi \sqrt{-1}}{p}\sum_{1\leq i<j \leq r}\widecheck{N}_{i,j}(l)h(u,p)_{i,j}} \mu(l)=\lim_{p\uparrow \infty}-\frac{1}{p^r} \log(\det(I-P^{A,U_{h(u,p)}}).$$\\
%Hence the characteristic function of $h_2(\mathcal{L}^{(2)}_{\alpha})$ equals $\lim_{p\uparrow \infty}\left[\frac{\det(I-P)}{(\det(I-P^{A,U_h(u,p)})^{\frac{1}{p^r}}}])\right]^{\alpha}$.
It follows that:

\begin{theorem}

 Denote by $\mathcal{L}_{\alpha}^{(>1)}$ the subset of $\mathcal{L}_{\alpha}$ formed by all loops of degree  greater than $1$. 
 Then for any set of integers $m_{i,j}, 1\leq i<j \leq r$, $\vert\{l\in \mathcal{L}^{(>1)}_{\alpha},\;\widecheck{N}_{i,j}(l)=m_{i,j}, 1\leq i<j \leq r \; \} \vert $ are independent Poisson r.v. with expectations:
 $$\alpha\mu(\{l,h_1(l)=0,\widecheck{N}_{i,j}(l)=m_{i,j}, 1\leq i<j \leq r \})=-\alpha
\int_{[0,1]^{\frac{r(r-1)}{2}}}F(u)\prod_{i<j} e^{-2\pi \sqrt{-1}m_{i,j}u_{i,j} }du_{i,j}.$$  with $F(u)=\lim_{p\uparrow \infty}\frac{1}{p^r}\log(\det(I-P^{(A,h(u,p))}))$.\\
%$\sqrt{\det(J)}$ being the volume of the Jacobian torus $\sqrt{\det[\delta_{i,j}C_{e_{i}}-C_{e_{i}}K_{e_i,e_j}C_{e_{j}}]}.
\end{theorem}

 \emph{Remarks:} \\

 a)  Consequently, the distribution of the second homology field $h_2(\mathcal{L}^{(>1)}_{\alpha})$ defined by $\mathcal{L}^{(>1)}_{\alpha}$ is given as follows. 
$P(\widecheck{N}_{i,j}(\mathcal{L}^{(>1)}_{\alpha})=m_{i,j} \; i=1...r)=$ is the limit as $p\uparrow \infty$ of:$$\int_{[0,1]^r}\int_{[0,1]^{\frac{r(r-1)}{2}}}\left[\frac{\det(I-P)}{(\det(I-P^{(\theta)})\det(I-P^{(A,h(u,p))})}\right]^{\alpha}) \prod_{i =1}^r\prod_{k<l} e^{-2\pi \sqrt{-1}m_{k,l}u_{k,l} }du_{k,l} d\theta_i.$$\\ 

b) The inverse Fourier transform can also be performed before taking the limit $p\uparrow \infty$.\\
 
 c) We could also express $F(u)$ as the logarithm of a Fredholm determinant, using the infinite dimensional representation $U_u$ of the group $G$ as operators on square integrable functions on the $r$-dimensional torus
  obtained by taking the limit as $p$ increases to infinity of $U_{h(u,p)}$:
  $$U_u[(a,c)]\phi (\theta)=e^{2\pi \sqrt{-1}(\left\langle c,u\right\rangle+\left\langle a,\theta\right\rangle )}\phi(\theta-u\cdot a)$$\\
 
d) The distributions of higher order homologies might be obtained in a similar way, but one would need to determine and use representations of nilpotent groups of higher class.\\

e) The relation with the signature of Brownian paths can be described as follows. A Brownian path segment in $\mathbb{R}^r$ can be thought of as the scaling limit of the lift to the universal  abelian cover of a long based random walk loop on a graph with  $r$-dimensional homology space, the endpoint of the path corresponding to the  loop first homology. Signature terms are scaling limits of the corresponding discrete signature terms. Conditioning on the sum of terms of degree $k<n$ to vanish,  the sum of terms of degree $n$ in the signature will always be shift invariant. In particular, Brownian loops are scaling limits of lifts of graph loops whose first homology vanishes. The $\frac{r(r-1)}{2}$ Levy areas of the Brownian loops are shift invariant and can be viewed as scaling limits of the coefficients of the second homology of these graph loops of vanishing first homology.
%The formulas obtained in this section are likely to be extendable to the Brownian case using zeta-regularized determinants. This was proved  in \cite{topo} for the first homology..  

%Acknowledgements.

\bigskip

\noindent
 NYU-ECNU Institute of Mathematical Sciences at NYU Shanghai. \\   
 D\'epartement de Math\'ematique. Universit\'e Paris-Sud.  Orsay
  
\bigskip
   yves.lejan@math.u-psud.fr \:
   yl57@nyu.edu


\begin{thebibliography}{11}

\bibitem{An} Nalini Anantharaman. Some relations between the spectra of simple and non-backtracking random walks.
arXiv:1703.03852. 


\bibitem{Chen}
Kuo-tsai Chen, Integration of paths--a faithful representation of paths
  by noncommutative formal power, Trans. Amer. Math. Soc. 156,
  395-407  (1971).
\bibitem{Ly}
Ben Hambly, Terry Lyons, Uniqueness for the signature of a path of
  bounded variation and the reduced path group, Ann. of Math. (2), 171
   109--167 (2010).
\bibitem {Kotasun}Motoko Kotani, Toshikazu Sunada. Zeta functions of finite graphs. J. Math. Sci. Univ. Tokyo 7, 7-25 (2000).
\bibitem{kosun} Motoko Kotani, Toshikazu Sunada. Jacobian Tori Associated with a Finite Graph and Its Abelian Covering Graphs. Advances in Applied Mathematics 24, 89–110. (2000).

\bibitem {LW}Gregory Lawler, Wendelin Werner. The Brownian loop soup. Probability Th. Rel. Fields 128, 565-588 (2004).
\bibitem{stfl} Yves Le Jan. Markov paths, loops and fields.  \'{E}cole d'\'{E}t\'{e} de Probabilit\'{e}s de Saint-Flour XXXVIII - 2008. Lecture Notes in Mathematics 2026. Springer. (2011).
 \bibitem{flownote} Yves Le Jan. Random flows defined by Markov loops. To appear in St\'{e}minaire  de Probabilit\'{e}s.

 \bibitem{lejanito} Yves Le Jan. Markov loops, free field and Eulerian networks.  arXiv:1405.2879.  J. Math. Soc. Japan 67, 1671-1680 (2015).


 \bibitem{MKS} Wilhelm Magnus, Abraham Karass, Donald Solitar. Dover. (1976) Combinatorial group theory. 2d Ed. Dover. (1976).
  \bibitem{lejanbakry} Yves Le Jan. Markov loops, Coverings and Fields. arXiv: 1602.02708.
 Annales de la Faculté des Sciences de Toulouse XXVI, 401-416 (2017).
 
 \bibitem{Mn} Pavel Mn\"ev. Discrete Path Integral Approach to the Selberg
Trace Formula for Regular Graphs. Comm. Math. Phys. 274, 233-241 (2007).

\bibitem {Mass}William S. Massey, \ Algebraic Topology: An Introduction \ Springer (1967).
\bibitem{Ree} Rimhak Ree.  Lie elements and an algebra associated with shuffles. Ann. of Math. (2)68 210–220, (1958).
\bibitem{Reu} Christophe Reutenauer. Free Lie algebras. Clarendon Press. Oxford. (1993).

\bibitem {Starter}Harold M. Stark, Audrey A. Terras, Zeta functions on finite graphs and
coverings. Advances in Maths 121, 134-165 (1996).

%%\bibitem {P}Jim Pitman, Combinatorial Stochastic Processes. 32th St Flour
%Summer School. Lecture Notes in Math.1875 Springer Berlin (2006)

%\bibitem{spitz} Frank Spitzer. Some theorems concerning 2-dimensional Brownian motion. Trans. Amer. Math. Soc. 87, 187–197 (1958).
%\bibitem{sle} Oded Schramm. Scaling Limits of Loop-Erased Random Walks and Uniform Spanning Trees. Israel J. Math.  Vol. 118, 221-288 (2000).
%\bibitem {ST}H.M. Stark, Audrey A. Terras, Zeta functions on finite graphs %and coverings. Advances in Maths 121 \ 134-165 (1996)
\bibitem {Symanz} Kurt Symanzik, Euclidean quantum field theory. Scuola
intenazionale di Fisica ''Enrico Fermi''. XLV Corso. 152-223 Academic Press. (1969);
%\bibitem {Yor} Marc Yor, Loi de l’indice du lacet brownien, et distribution de Hartman-Watson, Z. Wahrsch. Verw. Gebiete 53, 71–95 (1980).
%\bibitem {Zag} Don Zagier The Mellin Transformation and Other Useful Analytic Techniques. Appendix to Quantum Field Theory I, by E. Zeidler. Springer (2007).
%\bibitem{wernersemiprob}Wendelin Werner. On the spatial Markov property of soups of unoriented and oriented loops.  S\'{e}minaire de de Probabilit\'{e}s XLVIII pp. 481–503, LNM, volume 2168 Springer, (2016). 
\bibitem{WF} Yusuke Watanabe and Kenji Fukumizu, Graph zeta function in the Bethe free energy and
loopy belief propagation. Adv. Neural Inform. Proc. Sys. 22, 2017-2025.(2010).
\bibitem{Zag} Don Zagier. Appendix in Graphs on Surfaces and their Applications, by S.K. Lando and A.K. Zvonkin. Springer. (2004).
\end{thebibliography}
\end{document}